\documentclass[12pt, letterpaper]{amsart}
\usepackage{amsmath,amsfonts,amssymb, tikz, amsthm}
\usepackage[colorlinks=true,linkcolor=black,anchorcolor=black,citecolor=black,filecolor=black,menucolor=black,runcolor=black,urlcolor=blue]{hyperref}
\usepackage[utf8]{inputenc}
\usepackage[english]{babel}
\usepackage{xcolor}
\usepackage[tableposition=top]{caption} 
\usepackage{epigraph}
\usepackage{mathtools}
\usepackage{musicography}

\newtheorem*{theorem*}{Theorem}
\newtheorem*{proposition*}{Proposition}
\newtheorem{theorem}{Theorem}
\newtheorem{lemma}{Lemma}

\newtheorem{corollary}{Corollary}

\newtheorem{definition}{Definition}

\theoremstyle{remark}

\newtheorem{remark}{Remark}

\theoremstyle{proof}

\numberwithin{equation}{section}

\newcounter{para}
\newcommand{\p}{\par\refstepcounter{para}\noindent\textbf{\S\thepara}\space}
\numberwithin{para}{section}

\Alph{assumption}

\newcommand{\Z}{\mathbb{Z}}

\newcommand{\Q}{\mathbb{Q}}
\newcommand{\C}{\mathbb{C}}
\newcommand{\R}{\mathbb{R}}

\newcommand{\h}{\mathcal{H}}

\newcommand{\of}{\mathcal{O}_F}

\begin{document}
	
	\title{Determination of Hilbert modular forms using squarefree coefficients}
	
	\author{Rishabh Agnihotri}
	\email[Rishabh Agnihotri]{rishabhagnihotri663@gmail.com}
	\address{IISc, Bangalore}
	
	\author{Krishnarjun Krishnamoorthy}
	\email[Krishnarjun Krishnamoorthy]{krishnarjunmaths@outlook.com}
	\address{Beijing Institute of Mathemtical Sciences and Applications (BIMSA), No. 544, Hefangkou Village, Huaibei Town, Huairou District, Beijing.}

	\keywords{Half-integral weight Hilbert modular forms, Determination, Central values of twists, Shimura correspondence}
	\subjclass[2020] {11F41, 11F37, 11F27}
	
	\maketitle
	
	\begin{abstract}
		Let $F$ (over $\Q$) be a totally real number field of narrow class number $1$. We generalize a result of Kohnen on the determination of half-integral weight modular forms by their Fourier coefficients supported on squarefree (algebraic) integers. We also give a soft proof that infinitely many Fourier coefficients supported on squarefree integers are non-vanishing.
	\end{abstract}

	\section{Introduction}\label{Section "Introduction"}
	
	\p	Modular forms have established themselves as central objects of investigation in number theory due to the arithmetic nature of their ``Fourier coefficients''. Understanding the various properties of these Fourier coefficients comprises a significant part of modern day research. While it is clear that the entire collection of Fourier coefficients uniquely determine a modular form, it is interesting to ask for a ``minimal'' set of Fourier coefficients that would uniquely determine a modular form. Finding sharp answers to this question takes a particularly interesting turn in the context of half-integral weight modular forms.

	\p	Half-integral weight modular forms are analogues of the classical modular forms whose weight is a half-integer, prototypical examples being $\theta$-series of various kinds. There are two important developments in this theory that are relevant to this paper. The first one is the Shimura correspondence and the second one is Waldspurger's formula. Both these results provide (in a pleasant complementary fashion) a bridge between the half-integral realm and the realm of integral weight modular forms where more tools and techniques are available.
	
	\p	The Shimura correspondence reduces the study of half-integral weight modular forms to those Fourier coefficients supported on squarefree algebraic integers. The other coefficients are related to those of the Shimura lift in an explicit fashion. A deeper relation between the Fourier coefficients supported at squarefree integers and the Shimura lift is given by the Waldspurger formula. Therefore, philosophically, we expect the squarefree coefficients to carry all the information required to uniquely determine $f$. It is then natural to ask for a rigorous formulation of this expectation. Initial progress in this direction was due to Kohnen \cite{KohnenDetermination} when restricting to the ``plus space" (see \cite{KohnenPlus}, \cite{KohnenPlus2} for definitions).
	
	\begin{theorem}[Kohnen]\label{Theorem "Kohnen"}
		Let $f,g$ be two non-zero Hecke eigenforms in $S_{k+\frac{1}{2}}^+(4)$ with Fourier coefficients $a(n)$ and $b(n)$ respectively and suppose that $a(|D|) = b(|D|)$ for all fundamental discriminants $D$ such that $(-1)^kD > 0$. Furthermore suppose that $\lambda_2 = \mu_2$, where $\lambda_2$ (respectively $\mu_2$) denote the eigenvalues of $f$ (respectively $g$) under the Hecke operator $T_{k+\frac{1}{2}}^+(4)$. Then $f=g$.
	\end{theorem}
	
	\p	Kohnen mentions in \cite{KohnenDetermination} that the assumption on the $\lambda_2, \mu_2$ may be removed; he writes ``We expect that the additional condition on $\lambda_2$, in fact, is not necessary, but so far we cannot prove this''.	Modulo some technical subtleties, Kohnen's key idea is the fact that the functional equations of the natural Dirichlet series associated to a half-integral weight form and its Shimura lift have incompatible symmetries, but for the factors arising from the squarefree coefficients. Thus without the contribution of the squarefree coefficients, these incompatibilities may be utilized to our avail. This observation also underlies the proof of our main theorem (Theorem \ref{Theorem "Determination of half-integral weight forms modular forms"} below).	Kohnen treated a somewhat restricted case (of Hecke eigenforms in the plus space corresponding to the full level) but the methods are robust enough to be applied generally.
	
	\p	It is very convenient and quite tempting to restrict attention to newforms residing in the Kohnen plus space primarily because of the existence of a newform theory in this space and Kohnen's result which states that the newforms in the plus space is isomorphic to the space of integral weight newforms under the Shimura correspondence. In this light, the problem of determining half-integral weight forms by their squarefree coefficients is connected to determining integral weight newforms by the central values of the $L$ functions of their twists by quadratic characters, thanks to Waldspurger's formula\footnote{The determination using central values is a much general result as the Waldspurger formula does not take into account the sign of the Fourier coefficient of the half-integral weight form.}. In this direction, we have the following result of Luo-Ramakrishnan \cite{LuoRamakrishnan}.
	
	\begin{theorem}[Luo-Ramakrishnan]\label{Theorem "Luo-Ramakrishnan"}
		Let $f,g$ be normalized newforms in $S_{2k}(N), S_{2m}(N')$ respectively. Suppose that there is a constant $C$ such that
		\[
		L\left(\frac{1}{2},f\otimes \chi_d\right) = C L\left(\frac{1}{2}, g\otimes \chi_d\right)
		\]
		for almost all primitive quadratic characters $\chi_d$ of conductor $d$ prime to $NN'$, then $k=m$, $N=N'$ and $f=Cg$.
	\end{theorem}
	
	\p	This result has been generalized in many directions (see \cite{GangulyHoffsteinSengupta}, \cite{GMP}, \cite{HamiehTanabe}, \cite{MunshiMathAnnalen}, \cite{MunshiSengupta}, \cite{MunshiSenguptaForum}, \cite{PiJNT}). We remark that Theorem \ref{Theorem "Luo-Ramakrishnan"} allows us to compare forms of varying weights and levels. This feature is not present in Theorem \ref{Theorem "Kohnen"}. Here and henceforth we consider Hilbert modular forms over a totally real number field $F$ of narrow class number $1$. The purpose of this paper is to prove two results regarding the determination of half-integral weight Hilbert modular forms by their Fourier coefficients supported on squarefree integers.
		
	\begin{theorem}\label{Theorem "Determination of half-integral weight forms modular forms"}
		Suppose that $f$ and $g$ are half-integral weight cuspidal Hilbert Hecke eigenforms of varying weights and levels. Denote the normalized Fourier coefficients (see \eqref{Equation "lambda_f definition"}) of $f,g$ as $\lambda_f$ and $\lambda_g$ respectively. Suppose there exists $\kappa\neq 0$ such that
		\begin{equation}\label{Equation "Theorem 1"}
			\lambda_f(\tau) = \kappa \lambda_g(\tau)
		\end{equation}
		for every squarefree $\tau\in \of$. Then $f=\kappa g$.
	\end{theorem}
	
	\p	Theorem \ref{Theorem "Determination of half-integral weight forms modular forms"} can be seen as an intermediate result between Theorem \ref{Theorem "Kohnen"} and Theorem \ref{Theorem "Luo-Ramakrishnan"} in that, we allow for the flexibility of varying weights and levels, but still require an equality between the Fourier coefficients themselves and not just an equality between their sizes. Furthermore we require \eqref{Equation "Theorem 1"} to hold for \textit{every} squarefree coefficient, and not ``almost all''. It is unclear to us (and rather unlikely) if the present methods may be generalized to remove the above shortcomings.
	
	\p	It is desirable and often quite useful to obtain results which allow us to ``effectively'' determine the modular form. This amounts to a \textit{finite} number of conditions (as opposed to the requirement that \eqref{Equation "Theorem 1"} hold for every squarefree integer). In the case of Theorem \ref{Theorem "Luo-Ramakrishnan"}, this problem was resolved by Munshi \cite{MunshiMathAnnalen}. We present below a soft result in this direction in our setting.
	
	\begin{theorem}\label{Theorem "Non-vanishing of central values"}
		Suppose $f$ is a non-zero Hecke eigenform. Then 
		\[
		\underset{N(\tau) < T,\ \tau\ \mathrm{sq. free.}}{\limsup_{T\to\infty} } |\lambda_f(\tau)| \gg_{f,F} 1.
		\]
	\end{theorem}
	
	\p	Theorem \ref{Theorem "Non-vanishing of central values"} is slightly different from the $\Omega$ results obtained in \cite{GunKohnenOmega} and \cite{DasOmega} in that we restrict ourselves to squarefree integers. It is an interesting problem to show that the lower bound above is in fact $\infty$; that is to say that $\limsup\{|\lambda_f(\tau)|\} = \infty$. This would provide further credence to the generalization of Ramanujan conjecture as stated in \eqref{Equation "Ramanujan bound"} below. A recent result in this direction is available in \cite{GunKohnenSound}.
	
	\p	Theorem \ref{Theorem "Non-vanishing of central values"} is essentially a consequence of the Rankin-Selberg theory. The proof of Theorem \ref{Theorem "Non-vanishing of central values"} can be fine-tuned to give an upper bound on the least squarefree $\tau$ such that $|\lambda_f(\tau)|\neq 0$. This would amount to keeping track of the error term in the Rankin-Selberg estimate \eqref{Equation "Rankin-Selberg estimate"}.
	
	\p	The assumption of $F$ having narrow class number $1$ is made partly out of necessity and partly for convenience. Many parts of the proof can be modified with the weaker requirement that the discriminant of $F$ be a principal ideal, along with some other minor assumptions. We may further weaken our assumptions if we restrict to parallel weight forms. We feel that this setup best demonstrates the robustness of the new ideas incorporated into the proof of Theorem \ref{Theorem "Determination of half-integral weight forms modular forms"}.
	
	\subsection*{Acknowledgements}
	
	The authors would like to thank Winfried Kohnen for his encouragement and comments. K. K. was supported by Beijing Natural Science Foundation International Scientists Program (Grant No. IS24007). K. K. would also like to thank the Indian Insitute of Science for its generous hospitality.
			
	\section{Preliminaries}\label{Section "Preliminaries"}
	
	\p	\textbf{Notation:}	We let $F$ be a totally real extension of $\Q$ of degree $r$ and of narrow class number $1$. The various embeddings of $F$ into $\R$ will be labeled $\{\sigma_i\}_{i=1}^{r}$. For any $\sigma_i$ and $\xi\in F$, it will also be convenient to denote $\sigma_i(\xi)$ as $\xi_i$. Let $F^+$ denote the totally positive elements of $F$ and for any subset $S$ of $F$, the superscript $+$ will be used to denote the totally positive elements in $S$. That is $S^+ := S\cap F^+$. For any $\xi\in F$, the membership in $F^+$ will be denoted by $\xi\gg 0$. We shall denote the ring of integers of $F$ as $\of$ and its units as $U$. The different ideal of $F$ will be denoted by $\mathfrak{d}$ generated by $\delta\in \of$ and the discriminant of $F$ over $\Q$ will be denoted as $D_F$. Since we have assumed that $F$ has trivial narrow class group, every fractional ideal of $F$ has a totally positive generator and the totally positive units are squares; that is $U^+= \{u^2\ |\ u\in U\}$.
	
	\p	We adopt the following multi-index notation. We shall embed $F\hookrightarrow \R^r\hookrightarrow \C^r$ via its various embeddings $\{\sigma_i\}$. We shall also identify $\C\hookrightarrow \C^r$ via the diagonal embedding. Addition and multiplication in $\C^r$ is point-wise. Elements of $\C^r$ will generally be denoted by boldface letters, for example $\textbf{x},\ \textbf{y},\ \textbf{z}$ etc.	For $\textbf{w} = (w_1, \ldots, w_r),\textbf{z} = (z_1,\ldots,z_r)\in \C^r$ we let $\textbf{w}\cdot \textbf{z}\in \C$ be the sum $\sum_i w_iz_i$. For $\textbf{z} = (z_1,\ldots,z_r)\in \C^r$ and $\textbf{t} = (t_1,\ldots,t_r)\in (\C)^r$, we define $\textbf{z}^\textbf{t} := z_1^{t_1}\ldots z_r^{t_r}$. In the text, we shall take care that the exponentiation is unambiguous.
	
	\p	Even though we have assumed the narrow class number of $F$ to be equal to $1$, in some instances it is convenient to consider integral ideals of $\of$ rather than the algebraic integers themselves. We shall freely move between one setting and the other as our requirements may be. We warn the reader that in many cases, while working with integers in $\of$, we implicitly work modulo the unit group. We hope that this does not lead to any confusion.
	
	\p	\textbf{Hilbert modular forms:}	Standard references for the results that follow include \cite{GarrettBook}, \cite{ShimuraDuke78}, \cite{ShimuraDuke87}, \cite{ShimuraDuke93}. A weight $\textbf{k}$ is an element of $(\frac{1}{2}\Z)^r\subset \C^r$ of the form 
	\begin{equation}\label{Equation "weight definition"}
		\textbf{k} = \textbf{m}_{\textbf{k}} + \delta_\textbf{k}\left(\frac{1}{2}, \ldots, \frac{1}{2}\right),
	\end{equation}
	where $\textbf{m}_\textbf{k}\in (\Z)^r$ and $\delta_\textbf{k}\in \{0,1\}$. The theta series associated to $F$ is defined as
	\begin{equation}\label{Equation "Theta definition"}
		\theta_F(\textbf{z}) := \sum_{v\in \of} e^{\pi i  v^2\cdot \textbf{z}}.
	\end{equation}
	This is a holomorphic function on the generalized upper half plane $\h^r$, and is modular with respect to the group
	\begin{equation}
		\textbf{D}:= \left\{\gamma := \begin{pmatrix}
			a_\gamma&b_\gamma\\c_\gamma&d_\gamma
		\end{pmatrix} \ \Bigg| a_\gamma, d_\gamma\in \of, b_\gamma\in 2\mathfrak{d}^{-1}, c_\gamma\in 2\mathfrak{d}, \det(\gamma)=1 \right\}.
	\end{equation}	
	For $\gamma\in \textbf{D}$ and $\textbf{z}\in \h^r$, define
	\begin{equation}\label{Equation "Definition multiplier system"}
		h(\gamma, \textbf{z}):= \frac{\theta_F\left(\gamma\cdot \textbf{z}\right)}{\theta_F(\textbf{z})}.
	\end{equation}
	Furthermore, for $\gamma\in SL_2(F)$  we define $j\left(\gamma, z\right) = c_\gamma z+d_\gamma$ and let $j(\gamma, \textbf{z}) = \left(j(\sigma_i(\gamma), z_i)\right)\in \C^r$. A \textit{Hilbert modular cusp form} of weight \textbf{k} for a congruence subgroup $\Gamma$ of $SL_2(\of)$ is a holomorphic function $f : \h^r\to \C$ such that 
	\begin{enumerate}
		\item	$f(\gamma \textbf{z}) = h(\gamma, \textbf{z})^{\delta_\textbf{k}} j(\gamma, \textbf{z})^{\textbf{m}_\textbf{k}} f(\textbf{z})$ for every $\textbf{z}\in \h^r$ and $\gamma\in \Gamma$,
		\item 	$f$ vanishes at all the cusps of $\Gamma$.
	\end{enumerate}
	
	\p	We understand to allow $\delta_\textbf{k}=1$ only when $\Gamma \subseteq \textbf{D}$. Let $\Lambda$ be the collection of all $d\in F$ such that $\left(\begin{smallmatrix}1&d\\&1\end{smallmatrix}\right)\in \Gamma$. Then $f$ has a ``Fourier expansion" given by
	\[
	f(\textbf{z}) = \sum_{\xi \in \Lambda^\vee} a_f(\xi) e^{2\pi i \xi\cdot\textbf{z}}
	\]
	where $\Lambda^\vee$ denotes the dual of $\Lambda$ given by $\{\alpha\in F\ |\ \alpha \Lambda \subseteq \of\}$. The vanishing of $f$ at the cusps of $\Gamma$ implies that the Fourier expansion is supported only on totally positive elements $\xi$. We shall restrict our choices further so that the Fourier expansion is of the form 
	\begin{equation}\label{Equation "Fourier expansion"}
		f(\textbf{z}) = \sum_{\xi \in \of^+} a_f(\xi) e^{(2-\delta_\textbf{k})\pi i \xi\cdot\textbf{z}}.
	\end{equation}
	This corresponds to conventional choices of congruence subgroups in the integral and half-integral case. 
	
	\p	We may embed the unit subgroup $U$ inside $SL_2(\of)$ as $u\mapsto \left(\begin{smallmatrix}
		u&\\&u^{-1}
	\end{smallmatrix}\right)$ and, by abuse of notation, denote the image as $U$. For a congruence subgroup $\Gamma$, we let $U_\Gamma := U\cap \Gamma$, which is in particular a finite index subgroup of $U$. In the remainder of this paper, we shall choose congruence subgroups $\Gamma$ such that $U_\Gamma = U$. In particular, we define
	\begin{equation}\label{Equation "Gamma_0(n) definition"}
		\Gamma_0(n) := 	\left\{\gamma := \begin{pmatrix}
			a_\gamma&b_\gamma\\c_\gamma&d_\gamma
		\end{pmatrix} \ \Bigg| a_\gamma, d_\gamma\in \of, b_\gamma\in 2\mathfrak{d}^{-1}, c_\gamma\in 2^{-1}n\mathfrak{d}, \det(\gamma)=1 \right\}.
	\end{equation}
	We shall say that $f$ is of level $n$ if it is a modular form for the group $\Gamma_0(n)$; we shall implicitly require that $4|n$.
	
	\p	There is a commuting family of Hecke operators (denoted by $T_\mathfrak{n}$) indexed by the integral ideals $\mathfrak{n}$ of $\of$, which act on the space of modular forms that we are interested in. A cusp form $f$ of weight $\textbf{k}$ will be called a Hecke eigenform if it is an eigenvector for $T_{\mathfrak{n}^{1+\delta_k}}$ for all integral ideals $\mathfrak{n}$. We shall also suppose that Hecke eigenforms $f$ are normalized so that $T_{\of}f = f$ and denote the $\mathfrak{p}$-th eigenvalue as $\omega_\mathfrak{p}$, that is $T_{\mathfrak{p}^{1 + \delta_k}} f = \omega_\mathfrak{p} f$.
	
	\p\label{Subsection "L function well defined"}\textbf{The Dirichlet series associated to $f$:}	Suppose that $f$ is a modular form of weight $\textbf{k}$ with respect to some congruence subgroup $\Gamma$. Furthermore suppose that $f$ is a Hecke eigenform for all the Hecke operators with system of eigenvalues given by $\{\omega_\mathfrak{n}\}$. Then for any $u\in U_\Gamma$,
	\begin{equation}\label{Equation "Well-defined"}
		a_f(\xi u^2) = u^{\textbf{m}} a_f(\xi).
	\end{equation}	
	In particular, we may define the $L$-function associated to $f$ as
	\begin{equation}\label{Equation "L(s,f) definition"}
		L(s,f) := \sum_{\xi\in \of^+/U_\Gamma^+} \frac{\lambda_f(\xi)}{N(\xi)^s}
	\end{equation}
	where
	\begin{equation}\label{Equation "lambda_f definition"}
		\lambda_f(\xi) := a_f(\xi) \xi^{-\frac{\textbf{k}-1}{2}}.
	\end{equation}
	It is easy to see that this is well-defined. As in the classical case (see \cite[Section 1.7]{GarrettBook}), it can be shown that Fourier coefficients $a_f(\xi)$ satisfy the ``Hecke bound''
	\begin{equation}\label{Equation "Hecke bound"}
		a_f(\xi) \xi^{-\frac{\textbf{k}-1}{2}} \ll_\epsilon N(\xi)^{\frac{1}{2}+\epsilon}
	\end{equation}
	so that $L(s,f)$ is absolutely convergent for $\Re(s) > \frac{3}{2}$. It seems reasonable to expect (see \cite{GunKohnenSound}) that the Ramanujan conjecture in this case would be the bound 
	\begin{equation}\label{Equation "Ramanujan bound"}
		a_f(\xi) \xi^{-\frac{\textbf{k}-1}{2}} \ll_\epsilon N(\xi)^{\epsilon}.
	\end{equation}	
	We remark here that if $U_\Gamma = U$, for an integral ideal $\mathfrak{n}$ with a totally positive generator $n$, the definition $\lambda_f(\mathfrak{n}) := \lambda_f(n)$ is well-defined. In this fashion, we may rewrite $L(s,f)$ as a sum indexed by the integral ideals of $\of$.
	
	\p\textbf{The Shimura lift:}	For a half-integral weight Hilbert modular form $f$ of weight $\textbf{k}$ as above, we may associate an integral weight Hilbert modular form $\mathsf{f}$ of weight $2\textbf{m}$. More precisely, given a totally positive squarefree element $\tau$, let $\chi_\tau$ denote the Hecke character associated to the quadratic extension $F(\sqrt{\tau})/F$. Then the Shimura correspondence asserts that there exists a modular form of weight $2\textbf{m}$ (say $\mathsf{f}_\tau$) whose Fourier coefficients are described formally as
	\begin{multline}\label{Equation "Identity"}
		\underset{\xi\neq 0}{\sum_{\xi\in \of/U_\Gamma}} c(\xi, \mathsf{f}_\tau) M(\xi) = \left(\underset{\xi\neq 0}{\sum_{\xi\in \of/U_\Gamma}} a_f(\tau \xi^2) \xi^{-\textbf{m}} M(\xi)\right) \left(\underset{\xi\neq 0}{\sum_{\xi\in \of/U_\Gamma}} \frac{\chi_\tau(\xi) M(\xi)}{N(\xi)}\right)\\
		= a_f(\tau) \prod_{\mathfrak{p}} \left(1 - \omega_\mathfrak{p} M(\mathfrak{p}) + \frac{M(\mathfrak{p}^2)}{N(\mathfrak{p})}\right)^{-1}
	\end{multline}
	where $M(\xi) = M(\xi\of)$ is a formal symbol on integral ideals such that $M(\xi_1\xi_2) = M(\xi_1)M(\xi_2)$ and the product over $\mathfrak{p}$ is over all the prime ideals of $\of$. The second equality is valid only if $f$ is a Hecke eigenform (in which case $\mathsf{f}_\tau$ is also a Hecke eigenform). In this case $\mathsf{f}_\tau \not\equiv 0$ if and only if $a_f(\tau) \neq 0$. Moreover, whenever $a_f(\tau)\neq 0$, the modular form $\mathsf{f}_\tau a_f(\tau)^{-1}$ is independent of the choice of $\tau$. We shall denote this form as $\mathsf{f}$ and call it \textit{the} Shimura lift of $f$. Henceforth, we shall normalize Hecke eigenforms so that their first Fourier coefficient, that is $\lambda_{\mathsf{f}}(\of)$, equals unity.
	
	\p	\textbf{Rankin-Selberg theory:}	The Rankin-Selberg $L$-function of two half-integral weight forms $f,g$ (modular with respect to congruence subgroups $\Gamma_1, \Gamma_2$ respectively) is defined as 
	\begin{equation}\label{Equation "Rankin-Selberg definition"}
		L(s, f\times g) = \sum_{\xi \in \of^+/ U_\Gamma^+} \frac{\lambda_f(\xi)\overline{\lambda_g(\xi)}}{N(\xi)^s}
	\end{equation}
	where $\Gamma = \Gamma_1\cap \Gamma_2$. The series is absolutely convergent for $\Re(s) \gg 1$ and continues to a meromorphic function with a possible pole at $s=1$ \cite{ShimuraDuke87}. In particular, if $f=g$, then $L(s, f\times f)$ has a simple pole at $s=1$ with a positive residue. From a standard Mellin inversion argument, we can deduce the Rankin-Selberg estimate; as $T\to\infty$, we have
	\begin{equation}\label{Equation "Rankin-Selberg estimate"}
		\underset{N(\xi)\leqslant T}{\sum_{\xi \in \of^+/ U_\Gamma^+}} \left|\lambda_f(\xi)\right|^2 \sim \mathrm{Res}_{s=1} L(s,f\times f)T.
	\end{equation}
	
	\p	\textbf{Ray class groups:}	We briefly recall the notion of a ray class group. This section is quite general and so we may consider the case when $F$ is an arbitrary number field. A \textit{modulus} $\mathfrak{M}$ of $F$ is a (possibly empty) collection $\mathfrak{m}$ of real embeddings of $F$ along with an integral ideal $\mathfrak{b}$. Symbolically we write $\mathfrak{M} = \mathfrak{m}\mathfrak{b}$. A modulus $\mathfrak{M}_1 = \mathfrak{m}_1\mathfrak{b}_1$ is said to divide another modulus $\mathfrak{M}_2 = \mathfrak{m}_2\mathfrak{b}_2$ if $\mathfrak{m}_1\subseteq \mathfrak{m}_2$ and $\mathfrak{b}_1 | \mathfrak{b}_2$. 
	
	\p	Suppose $\mathfrak{M} = \mathfrak{m}\mathfrak{b}$ is a given modulus. Let $I(\mathfrak{b})$ denote the group of all fractional ideals of $F$ that are coprime to $\mathfrak{b}$. Consider the subgroup $P(\mathfrak{M})$ of all principal ideals $\{(\xi)\ |\ \sigma(\xi) > 0\ \forall \sigma\in \mathfrak{m}\ \&\ v_\mathfrak{p}(\xi-1)\geqslant v_\mathfrak{p}(\mathfrak{b}) \forall \mathfrak{p}|\mathfrak{b}\}$. Here $\mathfrak{p}$ stands for prime ideals and $v_\mathfrak{p}$ denotes the valuation at $\mathfrak{p}$. The quotient
	\begin{equation}\label{Equation "Ray class group definition"}
		Cl(\mathfrak{M}) := I(\mathfrak{b})/P(\mathfrak{M})
	\end{equation}
	is a finite group called the ray class group associated to the modulus $\mathfrak{M}$. When $\mathfrak{M} = \phi \of$, then the ray class group is the usual class group of $F$, which we denote by $Cl_F$. We remark that there is a natural map from $Cl(\mathfrak{M})\to Cl_F$ for any modulus $\mathfrak{M}$.
	\begin{lemma}\label{Lemma "Existence of modulus"}
		Given a finite collection of integral ideals $\{\mathfrak{a}_1,\ldots, \mathfrak{a}_n\}$ of $F$ and a set of real embeddings of $\mathfrak{m}$ of $F$, there exists an integral ideal $\mathfrak{b}$ coprime to all of $\mathfrak{a}_i$'s such that $\{[\mathfrak{a}_i]\}_{i=1}^{n}$ represent distinct classes in the ray class group modulo $\mathfrak{m}\mathfrak{b}$.
	\end{lemma}
	
	\begin{proof}
		Without loss of generality, we may suppose that $\mathfrak{m}=\phi$, since there is a natural map $Cl(\mathfrak{m}\mathfrak{b})\to Cl(\phi\mathfrak{b})$ for any integral ideal $\mathfrak{b}$. Therefore it suffices to prove the lemma for moduli of the form $\phi\mathfrak{b}$. For ease of notation, we denote the modulus $\phi\mathfrak{b}$ as simply $\mathfrak{b}$.
		
		 Let $S = \{\mathfrak{a}_1,\ldots, \mathfrak{a}_n\}$ be the set of given integral ideals. Define an equivalence relation on $S$ as $\mathfrak{a}_i\sim \mathfrak{a}_j$ if $[\mathfrak{a}_i] = [\mathfrak{a}_j]$ inside $Cl_F$. Let $S = \bigcup\limits_{j=1}^{k} S_j$ be a decomposition of $S$ into the corresponding equivalence classes. From class field theory, for any integral ideal $\mathfrak{b}$ we have the following exact sequence,
		\begin{equation}\label{Equation "Exact sequence"}
			1\to \of^\times \cap F^{\mathfrak{b},1} \to \of^\times \to F^{\mathfrak{b}}/F^{\mathfrak{b},1} \to Cl(\mathfrak{b}) \to Cl_F\to 1,
		\end{equation}
		where
		\[
		F^\mathfrak{b} = \{\xi\in F^\times\ |\ (\xi)\in I(\mathfrak{b})\}
		\]
		and
		\[
		F^{\mathfrak{b},1} = \{\xi\in F^\mathfrak{b}\ |\ (\xi)\in P(\mathfrak{b})\}.
		\]
		In particular there is a natural map from $Cl(\mathfrak{b})\to Cl_F$, and therefore, it suffices to show that $\mathfrak{b}$ can be chosen so that ideals inside each $S_i$ do not belong to the same class.
		
		For each $\mathfrak{a}_i, \mathfrak{a}_j\in S_k$, with $i\neq j$, we choose an element $\xi_{i,j,k}\in \of\setminus \of^\times$ such that either $\mathfrak{a}_i = \xi_{i,j,k}\mathfrak{a}_j$ or $\mathfrak{a}_j = \xi_{i,j,k}\mathfrak{a}_j$. Such elements exists by construction. Choose $\mathfrak{b}$ to be an integral ideal not containing the elements $\{ \xi_{i,j,k}-1\}_{i,j,k}$.	From the exactness of \eqref{Equation "Exact sequence"} at $F^{\mathfrak{b}}/ F^{\mathfrak{b},1}$, we see that $[(\xi_{i,j,k})]$'s represent nontrivial classes in $Cl(\mathfrak{b})$. The proof follows from here.
	\end{proof}
	
	\begin{lemma}\label{Lemma "Uniqueness of L functions"}
		Suppose $\lambda(\mathfrak{a})$ is an arithmetic function on the set of all integral ideals of $F$, so that the associated Dirichlet series	is absolutely convergent in some right half plane. Suppose that the twisted Dirichlet series
		\[
		L(s, \lambda\otimes \chi) := \sum_{\mathfrak{n}\subseteq \of} \frac{\lambda(\mathfrak{n}) \chi(\mathfrak{n})}{\mathcal{N}(\mathfrak{n})^s}
		\]
		vanishes identically for almost every Hecke character $\chi$. Then $\lambda\equiv 0$.
	\end{lemma}
	
	\begin{proof}[Proof of Lemma \ref{Lemma "Uniqueness of L functions"}]
		Let $\mathfrak{a}$ be a given integral ideal. Suppose $\{\mathfrak{a}=:\mathfrak{a}_1,\ldots,\mathfrak{a}_n\}$ be the complete collection of integral ideals whose norms equal that of $\mathfrak{a}_1$. From Lemma \ref{Lemma "Existence of modulus"} choose $\mathfrak{b}$ such that $\{[\mathfrak{a}_i]\}$ represent distinct classes in $Cl(\mathfrak{b})$. From assumption, the Dirichlet series
		\[
		\frac{1}{\left|Cl(\mathfrak{b})\right|} \sum_{\chi} \overline{\chi(\mathfrak{a}_1)} L(s, \lambda\otimes \chi) \equiv 0
		\]
		where the sum runs over all characters modulo $\mathfrak{b}$, viewed as Hecke characters. From the orthogonality of characters, the left hand side above equals
		\[
		\underset{[\mathfrak{n}] = [\mathfrak{a}_1]}{\sum_{\mathfrak{n}\subseteq \of}} \frac{\lambda(\mathfrak{n})}{\mathcal{N}(\mathfrak{n})^s}.
		\]
		From the choice of $\mathfrak{b}$ and  \cite[Exercise 9.1.12]{RamMurtyBook}, we may conclude that $\lambda(\mathfrak{a})=0$. This completes the proof as the choice was $\mathfrak{a}$ was arbitrary.
	\end{proof}
	
	\begin{remark}
		The above proposition for $F=\Q$ is quite trivial. We do not even the equality to hold for every twist. In fact we may directly employ \cite[Exercise 9.1.12]{RamMurtyBook}. Such an argument fails over number fields because we will have multiple integral ideals with the same norm. A simple minded Mellin inversion then becomes insufficient to conclude the above result.
	\end{remark}
	
	\section{Functional equation}\label{Section "Functional Equation"}	
	
	\p	\textbf{The Fricke involution:}	Recall the different ideal $\mathfrak{d}$ of $F$ is generated by $\delta$. Following \cite{SahaJNT}, we introduce the Fricke involution in our setting.	Since,
	\begin{equation}\label{Equation "Theta functional equation"}
		\theta_F\left(\textbf{z}\right) = (-\textbf{iz})^{-\frac{1}{2}} D_F^{-\frac{1}{2}} \theta_F\left(-\frac{1}{\delta^2 \textbf{z}}\right),
	\end{equation}
	we can extend the definition of $h(\gamma,\textbf{z})$ as
	\begin{equation}
		h\left(\begin{pmatrix}
			0&& - \delta^{-1}\\ \delta&&0
		\end{pmatrix}, \textbf{z}\right) := (-\textbf{iz})^{\frac{1}{2}} D_F^{\frac{1}{2}}.
	\end{equation}
	
	\p	Suppose $n\in \of^+$. Define the $W(n)$ operator as
	\begin{equation}
		W(n) := \left\{\begin{pmatrix}
			0 & -2\delta^{-1}\\ 2^{-1}\delta n & 0
		\end{pmatrix}, h(W(n), \textbf{z})\right\} =  \left\{\begin{pmatrix}
		0 & -2\delta^{-1}\\ 2^{-1}\delta n & 0
		\end{pmatrix}, (-\textbf{iz})^{\frac{1}{2}} \frac{D_F^{\frac{1}{2}}N(n)^\frac{1}{4}}{2^{\frac{r}{2}}}\right\}.
	\end{equation}
	Let $\chi_n$ denote the Hecke character associated to the quadratic extension $F(\sqrt{n})$ over $F$ as above. If $n$ is a perfect square in $F$, then we shall take $\chi_n$ to be trivial. 
	\begin{lemma}
		If $n$ is a totally positive integer and $f\in S_\textbf{k}(\Gamma_0(n))$, then $f|_{W(n)}\in S_\textbf{k}(\Gamma_0(n),\chi_n^{\delta_\textbf{k}})$.
	\end{lemma}
	
	\begin{proof}
		For ease of notation, let us write $W$ for $W(n)$. It is enough to verify the transformation property of $f|_W$, that is
		\[
		\left(f|_{W}\right)|_\gamma = \chi_n^{\delta_\textbf{k}}(a_\gamma) f|_{W}
		\]
		for every $\gamma= \left(\begin{smallmatrix}
			a_\gamma&b_\gamma\\c_\gamma&d_\gamma
		\end{smallmatrix}\right) \in\Gamma_0(n)$. For such a $\gamma$, we have
		\[
		W\gamma W^{-1} = \begin{pmatrix}d_\gamma&-2^2\delta^{-2}n^{-1}c_\gamma\\ -2^{-2}\delta^2 nb_\gamma&a_\gamma\end{pmatrix}=: \gamma'\in \Gamma_0(n).
		\]
		Then it follows that
		\[
		f|_W|_\gamma (\textbf{z}) = h(\gamma, \textbf{z})^{-\delta_\textbf{k}} j(\gamma, \textbf{z})^{-\textbf{m}} h(W, \gamma \textbf{z})^{-\delta_{\textbf{k}}} j(W, \gamma\textbf{z})^{-\textbf{m}} h(\gamma', W\textbf{z})^{\delta_\textbf{k}} j(\gamma', W\textbf{z})^{\textbf{m}} h(W,  \textbf{z})^{\delta_\textbf{k}} j(W, \textbf{z})^{\textbf{m}} f|_W(\textbf{z}).
		\]
		Thus, it suffices to show that 
		\[
		h(\gamma, \textbf{z})^{-\delta_\textbf{k}} j(\gamma, \textbf{z})^{-\textbf{m}} h(W, \gamma \textbf{z})^{-\delta_\textbf{k}} j(W, \gamma\textbf{z})^{-\textbf{m}} h(\gamma', W\textbf{z})^{\delta_\textbf{k}} j(\gamma', W\textbf{z})^{\textbf{m}} h(W,  \textbf{z})^{\delta_\textbf{k}} j(W, \textbf{z})^{\textbf{m}} = \chi_n^{\delta_\textbf{k}}(a_\gamma)
		\]
		Since the $j$ function satisfies the cocycle condition, we have
		\[
		j(\gamma', W\textbf{z})^{\textbf{m}}  j(W, \textbf{z})^{\textbf{m}}=j(\gamma, \textbf{z})^{\textbf{m}} j(W, \gamma\textbf{z})^{\textbf{m}}.
		\]
		This completes the proof when $\delta_\textbf{k}=0$. If $\delta_\textbf{k}=1$, we follow \cite[(4.0.1)]{SahaJNT}.
	\end{proof}
	
	\p	Following the above strategy, we may show that the $W(n)$ operator commutes with all the Hecke operators away from $n$ (see \cite[Lemma 4.1]{SahaJNT}). In particular, if $f$ is a Hecke eigenform, then so is $f|_{W(n)}$. Furthermore, using the strong multiplicity one theorem for $GL(2)$, we can conclude that the Shimura lifts of $f$ and $f|_{W(n)}$ coincide.
	\begin{lemma}\label{Lemma "Involution"}
		If we define
		\[
		f|_{W'(n)} := \textbf{i}^{\textbf{m}}n^{\frac{\textbf{m}}{2}} f|_{W(n)},
		\]
		then, $	f|_{W'(n)}|_{W'(n)} = f$.
	\end{lemma}
	
	\begin{proof}
		Follows by direct computation.
	\end{proof}
	
	\begin{theorem}\label{Theorem "Functional Equation"}
		Suppose $f$ is a half-integral weight cusp form of weight $\textbf{k}$ and level $n$. Then,
		\begin{enumerate}
			\item	$L(s,f)$ is absolutely convergent for $\Re(s) > 1$,
			\item 	continues analytically to the whole complex plane and satisfies the functional equation
			\[
			\Lambda(s, f) = \Lambda(1-s, f|_{W'(n)})
			\]
		\end{enumerate}
		where
		\[
		\Lambda(s,f) := D_F^s N(n)^{\frac{s}{2}} (2\pi)^{-rs} \prod_{j=1}^{r} \Gamma\left(s + \frac{k_j-1}{2}\right)L(s,f).
		\]
	\end{theorem}
	
	\begin{proof}
		Let us denote the abscissa of convergence of $L(s,f)$ as $\mathfrak{a}_f$.  From Cauchy-Schwartz inequality and \eqref{Equation "Rankin-Selberg estimate"}, 
		\begin{equation}\label{Equation "Abscissa inequality"}
			\underset{T\leqslant N(\xi)\leqslant 2T}{\sum_{\xi \in \of^+/ U_\Gamma^+}} \frac{\left|\lambda_f(\xi)\right|}{N(\xi)^\sigma} \ll_f T^{1-\sigma}
		\end{equation}
		In particular, if $\sigma > 1$, then
		\[
		\underset{N(\xi)\leqslant 2^n}{\sum_{\xi \in \of^+/ U_\Gamma^+}} \frac{\left|\lambda_f(\xi)\right|}{N(\xi)^\sigma} \ll_f \frac{1}{1 - 2^{1-\sigma}}
		\]
		Thus we may choose $\sigma = 1+\epsilon$ for any positive $\epsilon$ and hence conclude that $L(s,f)$ is convergent for $\Re(s) > 1$. Thus $\mathfrak{a}_f \leqslant 1$.
		
		We next prove the analytic continuation and functional equation. Set $T := (\R^+)^{r}$ and we may identify the totally positive elements $F^+\hookrightarrow T$ via its embeddings. From the unit theorem, the image of $U^+\hookrightarrow T$ is a degree $r-1$ lattice. We start with the integral
		\begin{equation}\label{Equation "I(s) definition"}
			I(s) := \int\limits_{T/U_\Gamma^+} f(\textbf{iy}) \textbf{y}^{s+\frac{\textbf{k}-1}{2}} \frac{d\textbf{y}}{\textbf{y}}.
		\end{equation}
		This is well defined from the arguments in \S \ref{Subsection "L function well defined"}. Furthermore the exponential decay of $f$ at the cusps ensures that the integral is convergent (and defines an analytic function) for $\Re(s)\gg 1$. Because of absolute convergence, for $\Re(s)$ large enough, we have
		\begin{align}\nonumber
			I(s) &= \int\limits_{T/U_\Gamma^+} \sum_{\xi \in \of^+} a_f(\xi) e^{-\pi \xi\cdot \textbf{y}} \textbf{y}^{s+\frac{\textbf{k}-1}{2}} \frac{d\textbf{y}}{\textbf{y}}\\ \nonumber
			&= \int\limits_{T} \sum_{\xi \in \of^+/ U_\Gamma^+} a_f(\xi) e^{-\pi \xi\cdot \textbf{y}} \textbf{y}^{s + \frac{\textbf{k}-1}{2}} \frac{d\textbf{y}}{\textbf{y}}\\\label{Equation "Functional equation LHS"}
			&= (\pi)^{-rs - \frac{\textbf{k}-1}{2}} \prod_{j=1}^{r} \Gamma\left(s + \frac{k_j-1}{2}\right) L(s, f).
		\end{align}	
		We break the integration into two parts. Suppose $\textbf{y} = (y_1,\ldots, y_r)$. We have
		\[
		I(s) = \int\limits_{\prod_i y_i \geqslant \beta} + \int\limits_{\prod_i y_i < \beta} =: I_1(s) + I_2(s).
		\]
		where $\beta := \frac{2^r}{D_F \sqrt{N(n)}} > 0$. The first integral is convergent for all $s\in \C$.	So we focus on the second integral. Changing the variable $\textbf{w} = (w_1,\ldots,w_r)$ such that $4 = \delta_i^2 n_i y_iw_i $, we have
		\[
		I_2(s) = \int\limits_{\prod_i w_i \geqslant \beta}  \left(\frac{4}{\delta^2 n\textbf{w}}\right)^{s + \frac{\textbf{k}-1}{2}} f\left(\frac{-4}{\delta^2 n \textbf{iw}}\right) \frac{d\textbf{w}}{\textbf{w}}.
		\]
		We observe that
		\[
		f\left(\frac{-4}{\delta^2 n \textbf{iw}}\right) = \left(\frac{D_F^2N(n)}{4^r}\right)^{\frac{1}{4}} \left(\frac{\delta n \textbf{iw}}{2}\right)^{\textbf{m}} \textbf{w}^{\frac{1}{2}} f|_W(\textbf{iw}).
		\]
		Therefore the previous integral becomes 
		\begin{align*}
			I_2(s) &= \left(\frac{D_F^2N(n)}{4^r}\right)^{\frac{1}{4}} \int\limits_{\prod_i w_i \geqslant \beta} f|_W(\textbf{iw}) \textbf{w}^{\frac{1}{2}}\left(\frac{\delta n \textbf{iw}}{2}\right)^{\textbf{m}} \left(\frac{4}{\delta^2 n\textbf{w}}\right)^{s + \frac{\textbf{k}-1}{2}} \frac{d\textbf{w}}{\textbf{w}}\\
			&= \textbf{i}^{\textbf{m}} \left(\frac{2}{\delta \sqrt{n}}\right)^{2s-1} n^{\frac{\textbf{m}}{2}} \int\limits_{\prod_i w_i \geqslant \beta} f|_W(\textbf{iw}) \textbf{w}^{\frac{\textbf{k}+1}{2} - s} \frac{d\textbf{w}}{\textbf{w}}.
		\end{align*}
		Observing that $f|_W$ is a cusp form, we may conclude that $I_2(s)$ is holomorphic on the whole complex plane. This makes $I(s)$ an analytic function on the whole complex plane. In order to prove the functional equation, we rewrite things in terms of the $W'$ operator;
		\[
		I_2(s) = \left(\frac{2}{\delta \sqrt{n}}\right)^{2s-1} \int\limits_{\prod_i w_i \geqslant \beta} f|_{W'}(\textbf{iw}) \textbf{w}^{\frac{\textbf{k}+1}{2} - s} \frac{d\textbf{w}}{\textbf{w}}.
		\]
		To conclude we have
		\[
		I(s) = I_1(s) + I_2(s) = \int\limits_{\prod_i y_i \geqslant \beta} f(\textbf{iy}) \textbf{y}^{s+\frac{\textbf{k}-1}{2}} \frac{d\textbf{y}}{\textbf{y}} + \left(\frac{2}{\delta \sqrt{n}}\right)^{2s-1} \int\limits_{\prod_i y_i \geqslant \beta} f|_{W'}(\textbf{iy}) \textbf{y}^{\frac{\textbf{k}+1}{2} - s} \frac{d\textbf{y}}{\textbf{y}}.
		\]
		If we define
		\begin{equation}\label{Equation "Lambda definition"}
				\Lambda(s,f) := 2^{\frac{\textbf{k}-1}{2}} \frac{D_F^s N(n)^{\frac{s}{2}}}{2^{rs}} I(s),
		\end{equation}
		then 
		\begin{equation}\label{Equation "Half-integral functional equation"}
			\Lambda(s,f) = \frac{D_F^s N(n)^{\frac{s}{2}} }{(2\pi)^{rs}} \prod_{j=1}^{r} \Gamma\left(s + \frac{k_j-1}{2}\right) L(s, f) = \Lambda(1-s,f|_{W'}).
		\end{equation}
		This completes the proof.
	\end{proof}
	
	\section{Proof of Theorem \ref{Theorem "Determination of half-integral weight forms modular forms"}}\label{Section "Determination - Half integral"}	
	
	\p	\textbf{An equality of Dirichlet series:}	Let $f,g$ be given as in the theorem. Denote the levels of $f,g$ as $4n_f, 4n_g$ and weights as $\textbf{k} = \textbf{m}+\frac{\textbf{1}}{\textbf{2}}, \textbf{l} = \textbf{n} + \frac{\textbf{1}}{\textbf{2}}$ respcetively. Our strategy is to first show that there exists a permutation $\rho$ of $\{1,\ldots,r\}$ such that $l_{\rho(i)} = k_i$ for every $i\in \{1,\ldots,r\}$, and to show that $N(n_f) = N(n_g)$. This will allow us to prove that $\Lambda(s,f)=\Lambda(s,g)$.
	
	\p	Without loss of generality, we may suppose that $\kappa = 1$ (where $\kappa$ is as in \eqref{Equation "Theorem 1"}). Let $\mathsf{f}$ be the Shimura lift of $f$. Recall that $\mathsf{f}$ is a primitive new form of weight $2\textbf{m}$ and level $n_\mathsf{f}$. The Dirichlet series attached to $\mathsf{f}$ is as in \eqref{Equation "L(s,f) definition"} with the weight $2\textbf{m}$ in place of $\textbf{k}$;
	\[
	L(s, \mathsf{f}) = \sum_{\xi \in \of^+/ U_\Gamma^+} \frac{a_\mathsf{f}(\xi) \xi^{-\textbf{m} + \frac{1}{2}}}{N(\xi)^s}.
	\]
	Since the Ramanujan conjecture is known for $\mathsf{f}$, $L(s, \mathsf{f})$ is absolutely convergent for $\Re(s) > 1$. Proceeding as in the proof of Theorem \ref{Theorem "Functional Equation"}, we can show that 
	\begin{equation}\label{Equation "Functional equation Integral"}
		\Lambda(s, \mathsf{f}) = \frac{D_F^s N(n_\mathsf{f})^{\frac{s}{2}} }{(2\pi)^{rs}} \prod_{j=1}^{r} \Gamma\left(s + m_j - \frac{1}{2}\right) L(s, \mathsf{f}) = (-1)^{\textbf{m}} \Lambda(1-s, \mathsf{f}|_{W'}),
	\end{equation}
	Observe that, since $\mathsf{f}$ is a Hecke eigenform, $\Lambda(s,\mathsf{f})$ is holomorphic and non-vanishing in the region $\Re(s) > 1$ and in the region $\Re(s) < 0$ because of the functional equation. Given an integral ideal $\mathfrak{n}$, there exists a totally positive element $\eta$ such that $\mathfrak{n} = \eta\of$. Let us denote $\lambda_\mathsf{f}(\mathfrak{n}) := a_\mathsf{f}(\eta)\eta^{-\textbf{m}+\frac{1}{2}}$.
	
	\p	Suppose that $\rho$ is a totally positive generator for a prime ideal $\mathfrak{p}$. The connection between $\omega_\mathfrak{p}$ and the Fourier coefficient $\lambda_{\mathsf{f}}(\rho)$ can be deduced from \cite[(2.22), (2.24)]{ShimuraDuke78} as
	\[
	\lambda_{\mathsf{f}}(\rho) N(\rho)^{\frac{1}{2}} = \omega_{\mathfrak{p}}.
	\]
	
	\p	If we choose $M(\xi)$ to be $N(\xi)^{\frac{1}{2}-2s}$ for some complex parameter $\Re(s)\gg 1$ and observe that $\omega_\mathfrak{p}$ is the Hecke eigenvalue of $\mathsf{f}$ with respect to $T_\mathfrak{p}$, \eqref{Equation "Identity"} gives
	\[
	\left(\sum_{0\neq \xi\in \of^+/U^+} \frac{a_f(\tau\xi^2) \xi^{-\textbf{m}+\frac{1}{2}}}{N(\xi)^{2s}}\right) \left(\sum_{0\neq\xi\in \of^+/U^+} \frac{\chi_\tau(\xi)}{N(\xi)^{\frac{1}{2}+2s}}\right)
	= a_f(\tau) \prod_{\mathfrak{p}} \left(1 - \frac{\lambda_\mathsf{f}(\mathfrak{p})}{N(\mathfrak{p})^{2s}} + \frac{1}{N(\mathfrak{p})^{4s}}\right)^{-1}
	\]
	which is infact an identity of Dirichlet series in the above region. A renormalization gives
	\[
	\lambda_f(\tau) L(2s, \mathsf{f}) = L\left(2s+\frac{1}{2}, \chi_\tau\right) \left(\sum_{0\neq \xi\in \of^+/U^+} \frac{\lambda_f(\tau \xi^2)}{N(\xi)^{2s}} \right),
	\]
	where $\lambda_f(\tau)$ is given as in \eqref{Equation "lambda_f definition"}. For $\Re(s) > 1$, $L(s, \chi_\tau)$ is given by an absolutely convergent Euler product and is non-vanishing in that region. Therefore
	\begin{equation}\label{Equation "Identity L function 2"}
		\frac{\lambda_f(\tau)}{N(\tau)^s}  L\left(2s+\frac{1}{2}, \chi_\tau\right)^{-1} L(2s, \mathsf{f}) =  \sum_{0\neq\xi\in \of^+/U^+} \frac{\lambda_f(\tau \xi^2)}{N(\tau\xi^2)^{s}} .
	\end{equation}
	
	\p	Summing over all the squarefree $\tau$ (modulo $U^+$) we get
	\begin{equation}\label{Equation "Identity for f"}
		L\left(s, f\right) = L(2s, \mathsf{f}) \left(\underset{\tau\ \mathrm{sq.\  free}}{\sum_{\tau\in \of^+/U^+}} \frac{\lambda_f(\tau)}{N(\tau)^{s}}   L\left(2s+\frac{1}{2}, \chi_\tau\right)^{-1} \right).
	\end{equation}
	Proceeding similarly for $g$, (whose Shimura lift we denote by $\mathsf{g}$) we arrive at
	\begin{equation}\label{Equation "Identity for g"}
		L\left(s, g\right) = L(2s, \mathsf{g}) \left(\underset{\tau\ \mathrm{sq.\  free}}{\sum_{\tau\in \of^+/U^+}} \frac{\lambda_g(\tau)}{N(\tau)^{s}}   L\left(2s+\frac{1}{2}, \chi_\tau\right)^{-1} \right).
	\end{equation}
	
	\p	Define
	\begin{equation}\label{Equation "C(s) definition"}
		C(s) := \frac{\Lambda(s,f)}{\Lambda(s,g)}.
	\end{equation}
	where $\Lambda(s,f)$ is given in \eqref{Equation "Lambda definition"}. We shall show that $C(s)$ is a constant. Suppose that $\lambda_f(\tau) = \lambda_g(\tau)$ for every squarefree $\tau$.	Denote the weight of $g$ as $\textbf{l} = \textbf{n} + \left(\frac{1}{2}, \ldots, \frac{1}{2}\right)$. We shall first show that there exists a permutation $\rho$ of the indices $\{1,2,\ldots,r\}$ such that $l_i = k_{\rho(i)}$ for every $1\leqslant i\leqslant r$. Our proof is via contradiction. From \eqref{Equation "Identity for f"} and \eqref{Equation "Identity for g"}, we have
	\begin{equation}
		C(s) =\left(\frac{N(n_f)}{N(n_g)}\right)^{\frac{s}{2}} \left(\prod_{i=1}^{r}\frac{\Gamma\left(\frac{2s + m_i-\frac{1}{2}}{2}\right)}{\Gamma\left(\frac{2s + n_i-\frac{1}{2}}{2}\right)}\right) \frac{L(2s, \mathsf{f})}{L(2s, \mathsf{g})}.
	\end{equation}
	
	\p	Applying the Legendre duplication formula to the Gamma factors, we get
	\begin{equation}\label{Equation "Temp"}
		C(s) = \left(\frac{N(n_f)}{N(n_g)}\right)^{\frac{s}{2}} \left(\frac{N(n_{\mathsf{g}})}{N(n_{\mathsf{f}})}\right)^{s} 2^{\textbf{n}-\textbf{m}} \left(\prod_{i=1}^{r} \frac{\Gamma\left(\frac{2s + n_i + \frac{1}{2}}{2}\right)}{\Gamma\left(\frac{2s + m_i + \frac{1}{2}}{2}\right)}\right) \frac{\Lambda(2s, \mathsf{f})}{\Lambda(2s, \mathsf{g})}.
	\end{equation}
	From \eqref{Equation "Half-integral functional equation"} and \eqref{Equation "C(s) definition"}, we see that $C(s) = C'(1-s)$ where 
	\[
	C'(s):= \frac{\Lambda(s, f|_{W'(n_f)})}{\Lambda(s, g|_{W'(n_g)})}.
	\]
	Plugging this in, we get
	\begin{multline}
		\left(\frac{N(n_f)}{N(n_g)}\right)^{\frac{s}{2}} \left(\frac{N(n_{\mathsf{g}})}{N(n_{\mathsf{f}})}\right)^{s}\left(\prod_{i=1}^{r} \frac{\Gamma\left(\frac{2s + n_i + \frac{1}{2}}{2}\right)}{\Gamma\left(\frac{2s + m_i + \frac{1}{2}}{2}\right)}\right) \frac{\Lambda(2s, \mathsf{f})}{\Lambda(2s, \mathsf{g})} =\\
				\left(\frac{N(n_f)}{N(n_g)}\right)^{\frac{1-s}{2}} \left(\frac{N(n_{\mathsf{g}})}{N(n_{\mathsf{f}})}\right)^{1-s} \left(\prod_{i=1}^{r} \frac{\Gamma\left(\frac{2-2s + n_i + \frac{1}{2}}{2}\right)}{\Gamma\left(\frac{2-2s + m_i + \frac{1}{2}}{2}\right)}\right) \frac{\Lambda(2-2s, \mathsf{f})}{\Lambda(2-2s, \mathsf{g})}.
	\end{multline}
	
	\p	Here we use the fact that the Shimura lift of $f$ and $f|_{W'}$ coincide, similarly for $g$. Without loss of generality, we may suppose that for every $1\leqslant i,j\leqslant r$, $m_i\neq n_j$. Furthermore, we may suppose that $m_1$ is the minimum of $\{m_i, n_j\}_{i,j=1}^{r}$. Substituting $2s = m_1 + \frac{5}{2}$, we get
	\begin{multline}
		\left(\frac{N(n_f)}{N(n_g)}\right)^{\frac{2m_1 + 5}{8}} \left(\frac{N(n_{\mathsf{g}})}{N(n_{\mathsf{f}})}\right)^{\frac{2m_1 + 5}{4}}\left(\prod_{i=1}^{r} \frac{\Gamma\left(\frac{n_i + m_1 + 3}{2}\right)}{\Gamma\left(\frac{m_i+m_1+3}{2}\right)}\right) \frac{\Lambda\left(m_1+\frac{5}{2}, \mathsf{f}\right)}{\Lambda\left(m_1+\frac{5}{2}, \mathsf{g}\right)} =\\
		\left(\frac{N(n_f)}{N(n_g)}\right)^{\frac{-2m_1-1}{8}} \left(\frac{N(n_{\mathsf{g}})}{N(n_{\mathsf{f}})}\right)^{\frac{-2m_1-1}{4}}\left(\prod_{i=1}^{r} \frac{\Gamma\left(\frac{n_i - m_1}{2}\right)}{\Gamma\left(\frac{m_i - m_1}{2}\right)}\right) \frac{\Lambda\left(-\frac{1}{2}-m_1, \mathsf{f}\right)}{\Lambda\left(-\frac{1}{2}-m_1, \mathsf{g}\right)}.
	\end{multline}
	This leads to a contradiction as the right hand side vanishes but the left hand side does not. Thus the choice for $m_1$ is not possible, whence it follows that there exists a permutation $\rho$ of the indices $\{1,2,\ldots,r\}$ such that $l_i = k_{\rho(i)}$ for every $1\leqslant i\leqslant r$. Thus
	\begin{equation}\label{Equation "Temp 1"}
		C(s) = \left(\frac{N(n_f)}{N(n_g)}\right)^{\frac{s}{2}} \frac{L(s, f)}{L(s, g)} =  \left(\frac{N(n_f)}{N(n_g)}\right)^{\frac{s}{2}} \frac{L(2s, \mathsf{f})}{L(2s, \mathsf{g})}.
	\end{equation}
	As $L(s, \mathsf{f})$ and $L(s,\mathsf{g})$ are given by an absolutely converging Euler product in the region $\Re(s) > 1$, we conclude that $C(s)$ is holomorphic and non-vanishing in the region $\Re(s) > \frac{1}{2}$. Furthermore, $C(s)$ continues meromorphically to the whole complex plane.

	\p	Define
	\begin{equation}
		C_1(s):= C(s)\left(\frac{N(n_g)}{N(n_f)}\right)^{\frac{s}{2}} \left(\frac{N(n_{\mathsf{f}})}{N(n_{\mathsf{g}})}\right)^{s}.
	\end{equation}
	From the functional equations of $L(s, \mathsf{f})$ and $L(s, \mathsf{g})$, we see $C_1(s) = \varepsilon_{\mathsf{f}} \varepsilon_{\mathsf{g}}^{-1} C_1\left(\frac{1}{2}-s\right)$, where $\varepsilon_{\mathsf{f}}, \varepsilon_{\mathsf{g}}\in \{\pm 1\}$ are the eigenvalues of $\mathsf{f}, \mathsf{g}$ respectively under the respective Fricke involutions. Rewriting this in terms of $C(s)$, we get
	\begin{align*}
		\left(\frac{N(n_g)}{N(n_f)}\right)^{\frac{s}{2}} \left(\frac{N(n_\mathsf{f})}{N(n_\mathsf{g})}\right)^{s}C(s) &= \varepsilon_{\mathsf{f}} \varepsilon_{\mathsf{g}}^{-1} \left(\frac{N(n_g)}{N(n_f)}\right)^{\frac{1}{4}-\frac{s}{2}} \left(\frac{N(n_\mathsf{f})}{N(n_\mathsf{g})}\right)^{\frac{1}{2}-s} C\left(\frac{1}{2}-s\right)\\
		&= \varepsilon_{\mathsf{f}} \varepsilon_{\mathsf{g}}^{-1} \left(\frac{N(n_g)}{N(n_f)}\right)^{\frac{1}{4}-\frac{s}{2}} \left(\frac{N(n_\mathsf{f})}{N(n_\mathsf{g})}\right)^{\frac{1}{2}-s} C'\left(s+\frac{1}{2}\right).
	\end{align*}
	Applying the recursive relation once again for $C'(s)$ and using Lemma \ref{Lemma "Involution"}, we get that
	\begin{equation}\label{Equation "Periodicity"}
		\left(\frac{N(n_g)}{N(n_f)}\right)^{s + \frac{1}{4}} \left(\frac{N(n_{\mathsf{f}})}{N(n_{\mathsf{g}})}\right)^{2s + \frac{1}{2}} C(s) = \left(\frac{N(n_g)}{N(n_f)}\right)^{\frac{1}{4} - s} \left(\frac{N(n_{\mathsf{f}})}{N(n_{\mathsf{g}})}\right)^{\frac{1}{2}-2s} C(s+1).
	\end{equation}
	
	\p	Rewriting the above equation, we get
	\begin{equation}\label{Equation "Temp 2"}
		\left( \left(\frac{N(n_g)}{N(n_f)}\right)^{\frac{s}{2}} C(s)\right) =
		\left(\frac{N(n_g)}{N(n_f)}\right)^{-\frac{1}{2} - 2s} \left(\frac{N(n_{\mathsf{f}})}{N(n_{\mathsf{g}})}\right)^{-4s} \left(\left(\frac{N(n_g)}{N(n_f)}\right)^{\frac{s+1}{2}}C(s+1)\right).
	\end{equation}	
	From the Dirichlet series expansion for $L(s, \mathsf{f})$ (and our normalization), we see that $\lim_{\sigma\to\infty} L(\sigma + it, \mathsf{f})=1$ for any $t\in \R$. Choose $t=0$. From \eqref{Equation "Temp 1"}, choosing $s=\sigma\in \R$ and letting $\sigma\to\infty$ in \eqref{Equation "Temp 2"} leads to a contradiction unless $N(n_f) = N(n_g)$ and $N(n_\mathsf{f}) = N(n_\mathsf{g})$. In particular,  \eqref{Equation "Periodicity"} gives us $C(s) = C(s+1)$, and hence it follows that $C(s)$ is an entire function. Therefore it suffices to show that $C(s)$ is absolutely bounded. 
	
	\p	From periodicity, it is sufficient to do this for $\Re(s) \gg 1$ and we shall suppose so. Fix $t\in \R$. Arguing as above we may choose $\sigma$ large enough so that $|L(\sigma + it, \mathsf{f})|\leqslant 2$ and $|L(\sigma + it, \mathsf{g})| \geqslant \frac{1}{2}$ implying for large enough $\sigma$, $|C(\sigma+it)| \leqslant 4$. From periodicity, the above bound holds for every $\sigma\in \R$. Since $t$ was arbitrary, the above bound holds uniformly in $\C$.
	Thus $C(s)$ is a constant as claimed. From our normalization, we have shown that $L(s,f) = L(s,g)$.
	
	\p	To conclude that $f\equiv g$, we proceed as follows. Suppose that $\chi$ is a Hecke character with conductor away from the levels $n_f$ and $n_g$. We observe that the above proof follows almost verbatim, and we may deduce that $L(s, f\otimes \chi) = L(s, g\otimes \chi)$. Now we employ Lemma \ref{Lemma "Uniqueness of L functions"} to complete the proof.
	
	\p	\textbf{Restricting to the plus space:}	Consider forms $f$ of level $4n_f$ where $n_f$ is odd (coprime to $2$) and squarefree. The Kohnen-``plus'' space is a subspace of cusp forms characterized by certain conditions on the Fourier coefficients. For modular forms over $\Q$, Kohnen was able to show certain remarkable properties of this plus space in connection to the Shimura lift (see \cite{KohnenPlus},\cite{KohnenPlus2}). Suppose that $\xi\in \of$ is not a square in $F$. A tuple $(\xi, \eta)\in \of\times F$ is called a \textit{discriminant} if 
	\begin{enumerate}
		\item	there exists $\omega_\xi\in F(\sqrt{\xi})$ such that $Tr_{K_\xi/F}(\omega_\xi)^2 - 4 N_{K_\xi/F}(\omega_\xi) = \xi$, and
		\item 	$\of\oplus \omega_\xi\eta\of$ is an order in $K_\xi$.
	\end{enumerate}
	Furthermore we say that $(\xi,\eta)$ is a \textit{fundamental discriminant} if $\of\oplus \omega_\xi\eta \of$ is the ring of integers of $F(\sqrt{\xi})$.
	
	\p	A discriminant $(\xi,\eta)$ corresponds to the Fourier coefficient $\lambda_f(\xi\eta^{2})$ of $f$. In particular we can restrict ourselves to discriminants of the form $(\xi,1)$ for which $\xi\in F^+$. We may abbreviate the notation for fundamental discriminant $(\xi,1)$ as $\xi$.
	
	\p	Choose and fix a unit $v$ (if it exists) such that $N(v)=-1$. For any weight $\textbf{m}\in \Z^r$, we let $u_\textbf{m}$ denote either $1$ or $v$ depending on whether $(-1)^\textbf{m}=\pm1$ respectively. Now, we may extend \cite[Definition 13.2]{HiragaIkedaCompositio} for our case as follows. In the following, let us abbreviate $u_{\textbf{m}_f}$ as $u_f$; similarly for $g$.
	
	\begin{definition}[Kohnen plus space]\label{Definition "Kohnen plus space"}
		Given a weight $\textbf{k} = \textbf{m} + \frac{\textbf{1}}{\textbf{2}}$, let $u_\textbf{m}$ be as above. The Kohnen plus space is defined as the space of all half-integral weight modular forms of  weight $\textbf{k}$ such that $\lambda_f(\tau)=0$ unless $u_\textbf{m}\tau \equiv \square \mod 4$.
	\end{definition}
	
	\begin{corollary}\label{Corollary "Half"}
		Suppose $f$ and $g$ are two newforms of weights $\textbf{k}_f = \textbf{m}_f + \frac{\textbf{1}}{\textbf{2}}$ and $\textbf{k}_g = \textbf{m}_g + \frac{\textbf{1}}{\textbf{2}}$ respectively, in the Kohnen plus space (of respective weights and levels). Suppose that $\lambda_f(u_f\xi) = \lambda_g(u_g\xi)$ for every fundamental discriminant $\xi$. Then $f=g$.
	\end{corollary}
	
	\begin{proof}
		If $1=u_f\neq u_g$, then $u_f\xi$ and $u_g\xi$ cannot both simultaneously be squares modulo $4$. This leads to a degenrate case. Hence we may suppose $u_f = u_g= u$.  Suppose $\xi\in u\of^+$ is squarefree. From Theorem \ref{Theorem "Determination of half-integral weight forms modular forms"}, it is enough to verify that
		\begin{equation}\label{Equation "Verify"}
			\lambda_f(u\xi) = \lambda_g(u\xi).
		\end{equation}
		We may suppose that $\xi \equiv \square \mod 4$; otherwise both sides of \eqref{Equation "Verify"} vanish. In particular, this means that $\xi$ is coprime to $2$. Then from \cite[Theorem 1.7]{CohenQuartic}, the relative discriminant of $F(\sqrt{\xi})$ over $F$ equals $\xi$. In this case \eqref{Equation "Verify"} follows from our assumption completing the proof.		
	\end{proof}
		
	\section{Proof of Theorem \ref{Theorem "Non-vanishing of central values"}}\label{Section "Non-vanishing of central values"}
	
	\p	In order not to clog up notation, in this section, $\tau\in \of^+/ U^+$ shall denote a squarefree algebraic integer. For Hecke eigenforms, choosing $M(\xi) = N(\xi)^{-s}$ in \eqref{Equation "Identity"} and using M\"obius inversion we get
	\[
	\lambda_f(\tau\xi^2) = \lambda_f(\tau) \sum_{\eta_1\eta_2 = \xi\mod U}\mu_F(\eta_1) \frac{\chi_\tau(\eta_1)}{N(\eta_1)} \lambda_\mathsf{f}(\eta_2),
	\]
	where $\mu_F$ denotes the analogue of the M\"obius function for $F$. For $T\gg 1$ we have
	\[
	\underset{N(\tau\xi^2)\leqslant T}{\sum_{\tau\xi^2\in \of^+/U^+}} \left|\lambda_f(\tau\xi^2)\right|^2 = \underset{ N(\tau\xi^2)\leqslant T}{\sum_{\tau\xi^2\in \of^+/U^+}} \left|\lambda_f(\tau)\right|^2 \left|\underset{\eta_1, \eta_2\mod U}{\sum_{\eta_1\eta_2 = \xi}} \mu_F(\eta_1) \frac{\chi_\tau(\eta_1)}{N(\eta_1)} \lambda_\mathsf{f}(\eta_2)\right|^2.
	\]
	
	\p	Using the Rankin-Selberg estimate \eqref{Equation "Rankin-Selberg estimate"} on the left hand side and the Ramanujan bound (which is known for $\mathsf{f}$) on the right hand side we obtain, for some $0 < \epsilon < 1/2$ and $A\gg 1$,
	\begin{align*}
		T &\ll_{f,\epsilon} \underset{N(\tau\xi^2)\leqslant T}{\sum_{\tau\xi^2\in \of^+/U^+}} \left|\lambda_f(\tau)\right|^2 \xi^\epsilon = \underset{N(\tau)\leqslant T}{\sum_{\tau\in \of^+/U^+}} \left|\lambda_f(\tau)\right|^2 \sum_{N(\xi)\leqslant \sqrt{\frac{T}{N(\tau)}}}\xi^\epsilon\\
		&\ll T^{\frac{1}{2}+\epsilon} \underset{\log(T)\leqslant N(\tau)\leqslant T}{\sum_{\tau\in \of^+/U^+}} \frac{\left|\lambda_f(\tau)\right|^2}{N(\tau)^{\frac{1}{2}+\epsilon}} + \mathcal{O}\left(T^{\frac{1}{2}+\epsilon}\log^A(T)\right)\\
		&\ll_F T \sup_{\log(T) \leqslant N(\tau) \leqslant T} \left\{|\lambda_f(\tau)|^2\right\}  + \mathcal{O}\left(T^{\frac{1}{2}+\epsilon}\log^A(T)\right).
	\end{align*}
	Cancelling $T$ from both sides completes the proof.
	
	\bibliographystyle{amsalpha}
	
	\bibliography{Bibliography}

\end{document}